\documentclass[11pt]{amsart}

 \usepackage{amsfonts,graphics,amsmath,amsthm,amsfonts,amscd,amssymb,amsmath,latexsym,multicol,
 mathrsfs}
\usepackage{epsfig,url}
\usepackage{flafter}
\usepackage{fancyhdr}
\usepackage{tikz-cd}
\usepackage{hyperref}
\usepackage{comment}
\hypersetup{colorlinks=true, linkcolor=black}

 \usepackage[matrix, arrow]{xy}

\DeclareMathOperator{\Supp}{Supp}

\DeclareMathOperator{\Bir}{Bir}

 \numberwithin{equation}{subsection}
 \numberwithin{footnote}{subsection}

 \newtheorem{lem}[subsection]{Lemma}
 \newtheorem{prop}[subsection]{Proposition}
 \newtheorem{thm}[subsection]{Theorem}
 \newtheorem{conj}[subsection]{Conjecture}

{
\theoremstyle{upright}

 \newtheorem{defn}[subsection]{Definition}

 \newtheorem{rem}[subsection]{Remark}

}

 \newcommand{\N}{\mathbb N}

 \newcommand{\Q}{\mathbb Q}
 \newcommand{\R}{\mathbb R}

  \newcommand{\Diff}{\text{Diff}}
 
 \newcommand{\rddown}[1]{\left\lfloor{#1}\right\rfloor} 

\title{Some Results about the Index Conjecture for log Calabi-Yau Pairs}
\thanks{2010 MSC:
	14J32, 
	14J45, 
14E30, 
14C20, 
}
\author{Yanning Xu}
\date{\today}
\begin{document}
\maketitle
\begin{abstract}
	We show some inductive statements for the index conjecture for log canonical Calabi-Yau pairs. Using it, we show that boundedness of log canonical index for log canonical Calabi Yau pairs with rational DCC coefficients in dimension 3. We also show boundedness of log canonical index for klt log Calabi-Yau pairs in dimension 4 with $B\neq 0$. 
\end{abstract}

\tableofcontents
\newpage


\section{Introduction}
We investigate the following conjecture, known as the index conjecture for log Calabi-Yau pairs, in this paper. 
\begin{conj}\label{conj-bound-dlt-index}
	Let $d$ be a natural number and $\mathfrak{R}\subset [0,1]$ be a finite set of rational numbers.
	Then there exists a natural number $n$ 
	depending only on $d$ and $\mathfrak{R}$ satisfying the following.  
	Assume $(X,B)$ is a projective pair such that 
	\begin{itemize}
		
		\item $(X,B)$ is lc of dimension $d$,
		
		\item $B\in \Phi(\mathfrak{R})$, that is, the coefficients of $B$ are in $\Phi(\mathfrak{R})$, and 
		
		\item $K_X+B\sim_\Q 0$
	\end{itemize}
	Then  $n(K_X+B)\sim 0$.
	
\end{conj}
\begin{rem}
	By \cite{HMX}, it suffices to prove the above conjecture when coefficients of $B$ belong to $\mathfrak{R}$, a finite set of rationals.
\end{rem}
The above conjecture has a natural generalisation to the so called semi-log canonical version.
\begin{conj}\label{conj-bound-sdlt-index}
	Let $d$ be a natural number and $\mathfrak{R}\subset [0,1]$ be a finite set of rational numbers.
	Then there exists a natural number $n$ 
	depending only on $d$ and $\mathfrak{R}$ satisfying the following.  
	Assume $(X,B)$ is a projective pair such that 
	\begin{itemize}
		
		\item $(X,B)$ is slc of dimension $d$,
		
		\item $B\in \Phi(\mathfrak{R})$, that is, the coefficients of $B$ are in $\Phi(\mathfrak{R})$, and 
		
		\item $K_X+B\sim_\Q 0$
	\end{itemize}
	Then  $n(K_X+B)\sim 0$.
	
\end{conj}

Note that the above conjecture hold in dimension 2 \cite{PSh-II}.  We will show that the above conjecture will imply the following statement in higher dimension. Conjecture \ref{conj-bound-dlt-index} can be broken down to the special cases below.

\begin{conj}\label{conj-absolute-index}
	let d be a natural number. Then there exists $n$ positive integer, depending only on $d$ such that if $X$ is a klt Calabi-Yau of dimension $d$, i.e. $K_X\sim_\Q0$, then $nK_X\sim 0$
\end{conj}

Notice that the above conjecture is really related to the mld near 1 for calabi-yau varieites.
\begin{conj}\label{conj-klt-index}
	let d be a natural number and $\mathfrak{R}$ be a finite set of rationals. Then there exists $n$ positive integer, depending only on $d,\mathfrak{R}$ such that if $(X,B)$ is a klt log Calabi-Yau of dimension $d$ with $B>0$, i.e. $K_X+B\sim_\Q0$, and $B\in \mathfrak{R}$ then $n(K_X+B)\sim 0$.
\end{conj}
Conjecture \ref{conj-absolute-index} together with Conjecture \ref{conj-klt-index} is called the index conjecture for klt log Calabi-Yau pairs.
\begin{conj}\label{conj-index-lc-no-klt}
		Let $d$ be a natural number and $\mathfrak{R}\subset [0,1]$ be a finite set of rational numbers.
	Then there exists a natural number $n$ 
	depending only on $d$ and $\mathfrak{R}$ satisfying the following.  
	Assume $(X,B)$ is a projective pair such that 
	\begin{itemize}
		
		\item $(X,B)$ is lc of dimension $d$ and not klt,
		
		\item $B\in \Phi(\mathfrak{R})$, that is, the coefficients of $B$ are in $\Phi(\mathfrak{R})$, and 
		
		\item $K_X+B\sim_\Q 0$
	\end{itemize}
	Then  $n(K_X+B)\sim 0$.
\end{conj}
Clearly Conjecture \ref{conj-bound-dlt-index} is equivalent to Conjecture \ref{conj-absolute-index} + Conjecture \ref{conj-klt-index} + Conjecture \ref{conj-index-lc-no-klt}. We have the following inductive results in this paper.
\begin{thm}\label{thm-lc-index-bounded}
	Assuming Conjecture \ref{conj-bound-sdlt-index} hold in dimension $\leq d-1$, then Conjecture \ref{conj-index-lc-no-klt} hold in dimension d. 
\end{thm} 

We have the following results.
\begin{thm}\label{thm-klt-index-induction}
	Assuming Conjecture \ref{conj-absolute-index} in dimension $\leq d$, then Conjecture \ref{conj-klt-index} hold in dimension $\leq d+1$. In particular, assuming Conjecture \ref{conj-absolute-index} in dimension $\leq d$, then the Index conjecture for klt log Calabi-Yau pairs (Conjecture \ref{conj-absolute-index}+\ref{conj-klt-index}) hold in dimension $\leq d$. 
\end{thm}

In order to relate the above conjectures, we have the following conjecture.

\begin{conj}\label{conj_bounded_bir_auto}
	Let $n,d $ be a natural number and let $(X,B)$ be $d$-dimensional projective pair such that $(X,B)$ is dlt and $n(K_X+B)\sim 0$. The $Bir(X,B)$ denote the B-birational automorphism group of $(X,B)$. Then there exists $m,N\in \N$ depending only on $n,d$ such that $\rho_m(Bir(X,B))$ has size bounded by $N$, where $\rho_m: Bir(X,B)\rightarrow H^0(X,m(K_X+B))$ denotes the natural action by pulling back sections.  
\end{conj}
\begin{rem}
	Note that this has been shown by \cite{Fuj1} and \cite{Fujino-Gongyo} that $\rho_m$ is finite. Here we need it to be bounded. The above conjecture is true for $d=1$ but it is open for $d\geq 2$.
\end{rem}
Using the above conjecture, we have the following results.
\begin{thm}\label{thm-1.9}
    Conjecture \ref{conj-bound-dlt-index} in dimension $d$ and Conjecture \ref{conj_bounded_bir_auto} in dimension $d-1$ implies conjecture \ref{conj-bound-sdlt-index} in dimension $d$.
\end{thm}

Putting all of these together, we have the following results, which is the main result of the paper.
\begin{thm}\label{thm-induction-index}
	Assuming Conjecture \ref{conj-absolute-index} hold in dimension $\leq d$ and Conjecture \ref{conj_bounded_bir_auto} in dimension $\leq d-2$, Conjecture \ref{conj-bound-dlt-index} hold in dimension $d$.\\
		Assuming Conjecture \ref{conj-absolute-index} hold in dimension $\leq d$ and Conjecture \ref{conj_bounded_bir_auto} in dimension $\leq d-1$, Conjecture \ref{conj-bound-sdlt-index} hold in dimension $d$.
\end{thm}
In particular, we have the following definitive results. 
\begin{thm}\label{thm-lc-index-dim-3-4}
	Conjecture \ref{conj-bound-dlt-index} hold in dimension 3.
\end{thm}

We also have the following partial results in dimension $4$.
\begin{thm}\label{thm-index-4}
	Conjecture \ref{conj-bound-dlt-index} hold in dimension $d=4$ if $(X,B)$ is klt and $B\neq 0$
\end{thm}
Finally, we finish with a remark on the treatment of Conjecture \ref{conj-absolute-index} and \ref{conj_bounded_bir_auto}.
\begin{rem}
	Currently Conjecture \ref{conj_bounded_bir_auto} hold only in dimension 1. It is related to boundedness of betti numbers for klt log Calabi-Yau pairs (see \cite{Fujino-Gongyo}). This problem seems currently out of reach using methods only in Birational Geometry. \\\\Conjecture \ref{conj-absolute-index} is well-known for dimension 1 and 2. There are 2 cases for this conjecture. \begin{itemize}
		\item One case is when $X$ has terminal singularities and it is well-known that Conjecture \ref{conj-absolute-index} hold when dimension $d\leq 3$ with $X$ terminal. The case when $X$ is terminal and $\dim X\geq 4$ is open. It is expected that analytic methods are more suitable to tackle this problem.
	\item The other case is when $X$ is not terminal. This case is equivalent to  he ACC of mld near 1 for klt log Calabi-Yau pairs by passing to a plt blowup. This case is well-known for X is a surface. It has recently been claimed in \cite{jiang1} that this holds in dimension 3.
\end{itemize}
\end{rem}

The structure of this paper is as following. In section 2, we will introduce the basic definitions and preliminaries. In section 3 we will prove the main results. 
\\\\
I would like to thank Professor Birkar for proposing this problem and answering numerous questions from me. This work owes a great deal of his influence.

\newpage
\section{Preliminary}

All the varieties in this paper are quasi-projective over a fixed algebraically closed field of characteristic zero and a divisor means an $\R$-divisor
unless stated otherwise. The set of natural numbers $\N$ is the set of positive integers.

\subsection{Divisors}
Let $X$ be a normal variety, and let $M$ be an $\R$-divisor on $X$. 
We denote the coefficient of a prime divisor $D$ in $M$ by $\mu_DM$. If every non-zero coefficient of 
$M$ belongs to a set $\Phi\subseteq \R$, we write $M\in \Phi$. Writing $M=\sum m_iM_i$ where 
$M_i$ are the distinct irreducible components, the notation $M^{\ge a}$ means 
$\sum_{m_i\ge a} m_iM_i$, that is, we ignore the components with coefficient $<a$. One similarly defines $M^{\le a}, M^{>a}$, and $M^{<a}$.

Now let $f\colon X\to Z$ be a morphism to a normal variety. We say $M$ is \emph{horizontal} over $Z$ 
if the induced map $\Supp M\to Z$ is dominant, otherwise we say $M$ is \emph{vertical} over $Z$.

\subsection{Pairs}
A \emph{sub-pair} $(X,B)$ consists of a normal quasi-projective variety $X$ and an $\R$-divisor 
$B$ such that $K_X+B$ is $\R$-Cartier. 
If the coefficients of $B$ are at most $1$ we say $B$ is a 
\emph{sub-boundary}, and if in addition $B\ge 0$, 
we say $B$ is a \emph{boundary}. A sub-pair $(X,B)$ is called a \emph{pair} if $B\ge 0$.

Let $\phi\colon W\to X$ be a log resolution of a sub-pair $(X,B)$. Let $K_W+B_W$ be the 
pulback of $K_X+B$. The \emph{log discrepancy} of a prime divisor $D$ on $W$ with respect to $(X,B)$ 
is $1-\mu_DB_W$ and it is denoted by $a(D,X,B)$.
We say $(X,B)$ is \emph{sub-lc} (resp. \emph{sub-klt})(resp. \emph{sub-$\epsilon$-lc}) 
if $a(D,X,B)$ $\ge 0$ (resp. $>0$)(resp. $\ge \epsilon$) for every $D$. When $(X,B)$ 
is a pair we remove the sub and say the pair is lc, etc. Note that if $(X,B)$ is a lc pair, then 
the coefficients of $B$ necessarily belong to $[0,1]$. Also if $(X,B)$ is $\epsilon$-lc, then 
automatically $\epsilon\le 1$ because $a(D,X,B)=1$ for most $D$.

\subsection{Generalized polarised pairs}\label{ss-gpp}
For the basic theory of generalized polarised pairs see [\cite{BZh}, Section 4].
Below we recall some of the main notions and discuss some basic properties.\\

A \emph{generalized polarised pair} consists of 
\begin{itemize}
	\item a normal variety $X'$ equipped with a projective
	morphism $X'\to Z$, 
	
	\item an $\R$-divisor $B'\ge 0$ on $X'$, and 
	
	\item a b-$\R$-Cartier  b-divisor over $X'$ represented 
	by some projective birational morphism $X \overset{\phi}\to X'$ and $\R$-Cartier divisor
	$M$ on $X$
\end{itemize}
such that $M$ is nef$/Z$ and $K_{X'}+B'+M'$ is $\R$-Cartier,
where $M' := \phi_*M$. We can define singularities similarly for generalised pair. For precise definition and standard facts about generalised pair, see \cite{B-Fano}.

\subsection{Canonical bundle formula}
An algebraic fiber space $f:X \rightarrow Z$ with a given log canonical divisor $K_X+B$ which is $\Q$-linearly trivial over $Z$ is called an \emph{lc-trivial fibration}. The following canonical bundle formula for lc-trivial fibration from [\cite{Amb}] or [\cite{Fujino-Gongyo}] will be used.  

\begin{thm}[\text{[\cite{Fujino-Gongyo}, Theorem 1.1], [\cite{Amb}, Theorem 3.3]}]\label{c-bdle-form}
	Let $f:(X,B) \rightarrow Z$ be a projective morphism from a log pair to a normal variety $Z$ with connected fibers, $B$ be a $\mathbb{Q}$-boundary divisor. Assume that $(X,B)$ is sub-lc and it is lc on the generic fiber and $K_X+B \sim_\Q 0/Z$. Then there exists a boundary $\mathbb{Q}$-divisor $B_Z$ and a $\mathbb{Q}$-divisor $M_Z$ on $Z$ satisfying the following properties.  
	
	(i). $(Z,B_Z+M_Z)$ is a generalized pair, 
	
	(ii). $M_Z$ is b-nef.
	
	(iii). $K_X+B \sim_\mathbb{Q} f^\ast (K_Z+B_Z+M_Z)$.
\end{thm}

Under the above notations, $B_Z$ is called the \emph{discriminant part} and $M_Z$ is the \emph{moduli part}. 

\subsection{Slc Pairs}
We will introduce the basic definition here. More details of slc and sdlt pairs will be introduced in more details later. First we introduce demi normal schemes as in \cite{Kol}. A \emph{demi-normal scheme} is a scheme that is S2 and normal crossing in codimensional 1. Let $\Delta$ be an effective
$\Q$-divisor whose support does not contain any irreducible components of the conductor of X.
The pair $(X, \Delta)$ is called a \emph{semi-log canonical pair} (an \emph{slc pair}, for short) if
\begin{itemize}
	\item $K_X + \Delta$ is $\Q$-Cartier, and
	
	\item $(X', \Theta)$ is log canonical, where $\pi : X' \rightarrow X$ is the normalization and $K_{X'}+ \Theta = \pi^*(K_X +\Delta)$.
\end{itemize}

Similarly one defines a \emph{semi-divisorial log terminal pair} (an \emph{sdlt pair}, for short). An slc pair $(X,B)$ is said to be \emph{sdlt} if the normalisation $(X',\Theta)$ is dlt in the usual sense and $\pi: X_i'\rightarrow X_i$ is isomorphism. Note that here we are using the definition in \cite{Fuj1} instead of \cite{Kol}. This is fine since we will only use semi dlt in the following setting. We remark that if $(X,B)$ is a usual dlt pair, then $(\rddown{B},\Diff(B-\rddown{B}))$ is semi-dlt. Also for sdlt pair $(X,B)$ it is clear that $(K_X+B)|_{X_i} = K_{X_i}+\Theta_i$, where $\Theta_i := \Theta|_{X_i}$. We also have the following lemma/example.
\begin{lem}[\cite{Fujino-Gongyo-1}, Example 2.6]\label{lem-lt-slc-adjunction}
	Let $(X,B)$ be $\Q$-factorial lc with $B$ a $\Q$-divisor. Let $S := \rddown{B}$ and assume  $(X,B-\epsilon S)$ is klt for $0<\epsilon<<1$. Then $(S, \Diff(B-S))$ is slc.
\end{lem}

\subsection{Theorems on ACC}
Theorems on ACC proved in [\cite{HMX}] are crucial to the argument of the main results.
\begin{thm}[ACC for numerically trivial pairs, \text{[\cite{HMX}, Theorem D]}]\label{ACC-2}
	Fix a positive integer $n$ and a set $I \subset [0, 1]$, which satisfies the
	DCC.
	Then there is a finite set $I_0 \subset I$ with the following property: \\
	If $(X, \Delta)$ is a log canonical pair such that \\
	(i) $X$ is projective of dimension $n$, \\
	(ii) the coefficients of $\Delta$ belong to $I$, and \\
	(iii) $K_X + \Delta$ is numerically trivial, \\
	then the coefficients of $\Delta$ belong to $I_0$.	
\end{thm}
\subsection{Complements}\label{ss-compl}
(1) 
We first recall the definition for usual pairs.
Let $(X,B)$ be a pair where $B$ is a boundary and let $X\to Z$ be a contraction. 
Let $T=\rddown{B}$ and $\Delta=B-T$. 
An \emph{$n$-complement} of $K_{X}+B$ over a point $z\in Z$ is of the form 
$K_{X}+{B}^+$ such that over some neighbourhood of $z$ we have the following properties:
\begin{itemize}
	\item $(X,{B}^+)$ is lc, 
	
	\item $n(K_{X}+{B}^+)\sim 0$, and 
	
	\item $n{B}^+\ge nT+\rddown{(n+1)\Delta}$.\\
\end{itemize}
(2) 
Now let $(X',B'+M')$ be a projective generalized pair with data $\phi\colon X\to X'$ 
and $M$ where $B'\in [0,1]$. 
Let $T'=\rddown{B'}$ and $\Delta'=B'-T'$. 
An \emph{$n$-complement} of $K_{X'}+B'+M'$ is of the form $K_{X'}+{B'}^++M'$ where 
\begin{itemize}
	\item $(X',{B'}^++M')$ is generalized lc,
	
	\item $n(K_{X'}+{B'}^++M')\sim 0$ and $nM$ is b-Cartier, and 
	
	\item $n{B'}^+\ge nT'+\rddown{(n+1)\Delta'}$.\\
\end{itemize}

\subsection{Fano Varieties}
Let $(X,B)$ be a pair and $X\to Z$ a contraction. We say $(X,B)$ is \emph{log Fano} over $Z$ 
if it is lc and $-(K_X+B)$ is ample over $Z$; if $B=0$ 
we just say $X$ is Fano over $Z$. The pair is called \emph{weak log Fano} over $Z$ if it is lc 
and $-(K_X+B)$ is nef and big over $Z$; 
if $B=0$ we say $X$ is \emph{weak Fano} over $Z$.
We say $X$ is \emph{of Fano type} over $Z$ if $(X,B)$ is klt weak log Fano over $Z$ for some choice of $B$;
it is easy to see this is equivalent to existence of a big$/Z$ $\Q$-boundary (resp. $\R$-boundary) 
$\Gamma$ so that $(X,\Gamma)$ is klt and $K_X+\Gamma \sim_\Q 0/Z$ (resp. $\sim_\R$ instead of $\sim_\Q$). 

We have the following theorem on complements from [Birkar, \cite{B-Fano}].

\begin{thm}[\text{[\cite{B-Fano}, Theorem 1.10]}]\label{t-bnd-compl}
	Let $d$ and $p$ be natural numbers and $\mathfrak{R}\subset [0,1]$ be a finite set of rational numbers.
	Then there exists a natural number $n$ 
	depending only on $d,p$, and $\mathfrak{R}$ satisfying the following. 
	Assume $(X',B'+M')$ is a projective generalized polarised pair with data $\phi\colon X\to X'$ and $M$ 
	such that 
	\begin{itemize}
		\item $(X',B'+M')$ is generalized lc of dimension $d$,
		
		\item $B'\in \Phi(\mathfrak{R})$ and $pM$ is b-Cartier,
		
		\item $X'$ is of Fano type, and 
		
		\item $-(K_{X'}+B'+M')$ is nef.\\
	\end{itemize}
	Then there is an $n$-complement $K_{X'}+{B'}^++M'$ of $K_{X'}+{B'}+M'$ 
	such that ${B'}^+\ge B'$. 
\end{thm}

We also have the following theorem by Birkar on fano type fibrations that we will need in this paper. 
\begin{defn}[\cite{B-18}, Def 1.1]
Let $d,r$ be natural numbers and $\epsilon$ be a positive real number. A $(d,r,\epsilon)$ Fano type (log Calabi-Yau) fibration consists of a pair $(X,B)$ and a contraction $f:X\rightarrow Z$ such that we have the following.
\begin{itemize}
	\item $(X,B)$ projective $\epsilon$-lc of dimension $d$,
	\item $K_X+B\sim_\R f*L$ for some $\R$ divisor L,
	\item $-K_X$ is big over $Z$, i.e. $X$ is fano type over $Z$,
	\item A is a very ample divisor on Z with $A^{\dim Z}\leq r$,
	\item $A-L$ is ample
\end{itemize}
\end{defn}
\begin{thm}[\cite{B-18} Thm 1.2, Thm 1.3 ]\label{thm-B-bounded-CY}
	Let $d,r$ be natural numbers and $\epsilon,\delta$ be a positive real number. Consider the set of all $(d,r,\epsilon)$-Fano type fibration as above. Then the X form a bounded family. In particular if $B\geq \delta$, then the set of such $(X,B)$ form a log bounded family.
\end{thm}
\subsection{Calabi-Yau pairs}
Let $(X,B)$ be a pair, we say $(X,B)$ is log Calabi-Yau if $K_X+B\sim_\Q 0$ and $(X,B)$ is log canonical. the global index of $(X,B)$ is the minimal $n$ such that $n(K_X+B)\sim 0$. We have the following recent result in \cite{jiang1} about index for klt Calabi-Yau varieties.
\begin{thm}[\cite{jiang1}, Theorem 1.7]
	There exits $m\in \N$ such that if $X$ is projective klt variety of dimension 3 and $K_X\sim_\Q0$, then $mK_X\sim 0$.
	
\end{thm}
\subsection{B-Birational Maps and B-pluricanonical Representations}
We intoduce the notion of \emph{B-birational} as in \cite{Fuj1}. Let $(X,B)$, $(X',B')$ be  sub pairs, we say $f: (X,B)\dashrightarrow (X',B')$ is \emph{B-birational} if there is a common resolution $\alpha: (Y,B_Y)\rightarrow (X,B)$, $\beta:(Y,B_Y)\rightarrow (X,B)$ such that $K_Y+B_Y=\alpha^*(K_X+B)=\beta^*(K_{X'}+B')$ and a commuting diagram as the following.
\[
\begin{tikzcd}
&(Y,B_Y)\arrow[ld,"\alpha"'] \arrow[rd,"\beta"]&\\
(X,B)\arrow[rr,dashed,"f"]&&(X',B')
\end{tikzcd}
\] Let $$\Bir(X,B):=\{f|f:(X,B)\dashrightarrow (X,B) \text{\;is B-birational}\}$$ Let $n$ be a positive integer such that  $n(K_X+B)$ is Cartier. Then we define $$\rho_n: \Bir(X,B)\rightarrow Aut(H^0(X,n(K_X+B)))$$ be the representaion of the natural action of $\Bir(X,B)$ acting on $H^0(X,n(K_X+B))$ by pullbacking back sections.

We have the following theorem on B-birational representations.
\begin{prop} [\cite{Fujino-Gongyo-1} ]\label{prop-finiteness-bir-representation}
	Let $(X,B)$ be a projective (not necessarily connected) dlt pair with $n(K_X+B)\sim 0$ and $n$ is even, then $\rho_n(\Bir(X,B)) $ is finite. \qed
\end{prop}

\newpage

\section{Proof of main theorems}
 \subsection{Proof of Theorem \ref{thm-lc-index-bounded}}
\begin{proof}[Proof of Theorem \ref{thm-lc-index-bounded}]
Since $(X,B)$ is lc but not klt, it suffices to assume that, after replacing $(X,B)$ with a $\Q$-factorial dlt model, $(X,B)$ is dlt and $\rddown{B}\neq 0$. Choose $\epsilon>0$ small rational number and run a MMP on $K_X+B-\epsilon \rddown{B}\sim_\Q -\epsilon \rddown{B}$. Since it is not pseudo-effective, we will end up with a Mori-Fiber space $f: X'\rightarrow V$ with $dim V< \dim X$ . Now replacing $X$ by $X'$, we can assume $X=X'$. Now we consider different cases of dimension of $V$.
	\begin{enumerate}
		\item $\dim V=0$: In this case, X is a Fano type since $K_X+B-\epsilon\rddown{B}$ is anti-ample globally and $(X,B-\epsilon \rddown{B})$ is klt by construction. By boundedness of complements applying on $K_X+B$, we see that there is a bounded n depending only on $d,\mathfrak{R}$ such that $n(K_X+B)\sim 0$.
		\item $\dim V= \dim X-1$: In this case, we see that general fiber of $f: X\rightarrow V$ is a rational curve. By canonical bundle formula and since $X$ is fano type over $V$, there exists $\mathfrak{S}\subset[0,1]$ finite and $q$ depending only on $d,\mathfrak{R}$ such that $$q(K_X+B)\sim qf^*(K_V+B_V+M_V)$$ such that $B_V\in \mathfrak{S}$ and $qM_V$ is b-Cartier. Note that since $\dim V = \dim X-1$, by \cite{PSh-II}, we can assume that, by replacing $q$ by a bounded multiple, we have $qM_V$ is b-effective-base-point-free. Hence by possibly replacing $\mathfrak{S}$, we can write  $$q(K_X+B)\sim qf^*(K_V+B_V)$$ where $B_V\in \mathfrak{S}$. Now applying Conjecture \ref{conj-bound-dlt-index} in lower dimension, we get there exists a bounded $n$ depending only on $\mathfrak{S}$, which depends only on $d,\mathfrak{R}$, such that $n(K_V+B_V)\sim 0$. Therefore $nq(K_X+B)\sim 0$, which proves the result.
		\item $\dim V >0$ and $\dim V \leq \dim X-2$: In this case, the general fiber has dimension $\geq 2$. Now since $\epsilon\rddown{B}$ is ample over $V$, there exists $S\in \rddown{B}$ that is horizontal over $V$. By divisorial adjunction, we can write $$K_S+B_S=(K_X+B)|_S$$. Note that here $S$ may not be normal, but $(S,B_S)$ is slc by Lemma \ref{lem-lt-slc-adjunction}.  Note that coefficients of $B_S$ lies in a finite set that depends only on $\mathfrak{R}$. Firstly, we argue that $S\rightarrow V$ is a contraction. \\\\
		By stein-factorization, we can write $S\xrightarrow{f'} V'\xrightarrow{g} V$, where $f'$ is contraction and $g$ is finite. Since $V$ is normal, it suffices to show degree of $g$ is one. Assume that degree of $g$ is $m>1$. Let $x\in V$ be a general close point and let $F$ be the general fiber above $x$ in $X\rightarrow V$. Say $x_1,x_2,\dots, x_m$ are pre-image of $x$ in $V'$ and let $G_i$ be the fiber of $x_i$ in $S$. Note that $G_i$ are all disconencted. Now $S|_F = \cup G_i$, note that $S|_F$ is a well-defined since $F$ is a general fiber and $S$ is horizontal over $V$. However, $S|_F$ is ample divisor on $F$ by construction. Since $\dim F\geq 2$, all ample divisors on $F$ need to be connected, which is a contradiction.\\\\
		Now we also have by abuse of notation, we can also call $f:S\rightarrow V$. Let $V_sm$ be the smooth locus of $V$. Then we can assume $q(K_V+B_V+M_V)|_{V_{sm}}$ is Cartier. hence we $q(K_X+B)|_{f^{-1}(V_{sm})}$ is Cartier by canonical bundle formula. Now restricting to $S$, (since $S$ is horizontal), we have $q(K_S+B_S)|_{f^{-1}(V_{sm})}$ is also Cartier and we have $$q(K_S+B_S)|_{(f^{-1}(V_{sm}))}\sim f^*(q(K_V+B_V+M_V)|_{V_{sm}})$$ By induction (and possibly replacing $q$ by a bounded multiple), we can assume that $q(K_S+B_S)\sim 0$. Hence there is a rational function $\alpha$ such that $q(K_S+B_S) = div(\alpha)$. We can assume that $div(\alpha)$ is vertical over $V$. Using that $V$ is normal and $f_*\mathcal{O}_S=\mathcal{O}_V$, we have $\alpha|_{f^{-1}(V_{sm})}=\beta_{sm}\circ f$ for some rational function $\beta_{sm}$ on $V_{sm}$. Hence we get $q(K_V+B_V+M_V)|_{V_{sm}}\sim 0$ via $beta$, and therefore $q(K_V+B_V+M_V)\sim 0$ (since $V$ is normal and $V\backslash V_{sm}$ has codimension at least 2).  Hence we have $q(K_X+B)\sim 0$.
	\end{enumerate}
\end{proof}

\newpage
\subsection{Proof of Theorem \ref{thm-klt-index-induction}}
Here we need to use the following Theorem from \cite{Roberto-16}.
\begin{thm}[\cite{Roberto-16}, Thm 3.2]\label{thm-Roberto-CY-tower} 
	Let $(X,B)$ be a klt Calabi-yau pair with $B>0$. Then there exists a birational contraction $\pi: X\dashrightarrow X'$ to a klt Calabi-Yau pair $(X',B' := \pi_* B)$, $B'> 0$ and a tower of morphisms $$X'=X_0\xrightarrow{p_0}X_1\xrightarrow{p_1}X_2\xrightarrow{p_2}\dots \xrightarrow{p_{k-1}}X_k$$  with $k\geq 1$ such that \begin{itemize}
		\item for any $1\leq i <k$ there exists a boundary $B_i\neq 0$ on $X_i$ and $(X_i,B_i)$ is a klt Calabi-Yau Pair.
		\item for any $0\leq i\leq k$, the morphism $p_i$ is a $K_{X_i}$ Mori fiber space, and 
		\item either $\dim X_k =0$, i.e. $X_k=pt$ or $\dim X_k>0$ and $K_{X_k}\sim_\Q 0$
	\end{itemize} 
\end{thm}
Now it follows easily from [\cite{B-18}, Theorem  1.4] that we have the following. 
\begin{prop}
	Let $d$ be a natural number and $\mathfrak{R}$ be a finite set of rationals. Let $\mathcal{F}$ be the set of $(X,B)$, dimension $d$ klt Calabi-Yau pair (i.e. $(X,B)$ klt with $K_X+B\sim_\Q 0$) with $B\in \mathfrak{R}$ such that the corresponding $X_k=pt$ as in above theorem. Let $\mathcal{F'}$ the corresponding set of pairs of the form $(X',B')$. Then $\mathcal{F'}$ forms a bounded family hence there is a $n$ depending only on $d,\mathfrak{R}$ such that $n(K_X+B)\sim 0$ for any $(X,B)\in \mathcal{F}$. \qed
\end{prop}
Now we are ready to prove Theorem \ref{thm-klt-index-induction}.
\begin{proof}[Proof of Theorem \ref{thm-klt-index-induction}]
	By above proposition, it suffices to consider $(X,B)$ such that after applying Theorem \ref{thm-Roberto-CY-tower}, we end up with $\dim X_k >0$ and $K_{X_k}\sim_\Q 0$. Note that since we only care about the index, we can replace $X$ with $X'$. We will denote $p:X\rightarrow Z := X_k$ to be the composition of all the $p_i$. \\\\
	Firstly, we note that $p$ is a contraction. Let $(F,B_F:= B|_F)$ be the general fiber of $p$. Firstly we note that Restricting the morphism $p_i$ to $F$, we can deduce that $(F,B_F)$ belong to a bounded family depending only on $d,\mathfrak{R}$ using [\cite{B-18}, Theorem 1.4]. Hence there exists $r$ depending only on $d,\mathfrak{R}$ such that $r(K_F+B_F)\sim 0$. Hence we can apply canonical bundle formula and get $$r(K_X+B)\sim rf^*(K_Z+B_Z+M_Z)$$. Hence we can deduce that $K_Z+B_Z+M_Z\sim_\Q 0$. However $K_Z\sim_\Q 0 $. This implies that $B_Z =0 $ and $M_Z\equiv 0$. Now we can apply [\cite{floris}, Thm 1.3], to deduce that there exists $m$ depending only on $d,\mathfrak{R}$ such that $mK_Z\sim 0$. Writing $n:= mr$. Hence we deduce that $(K_X+B)\sim nf^*K_Z$. Now since $\dim Z<\dim X$, by hypothesis, we know that there is a bounded $q$ depending only on $d$, such that $qK_Z\sim 0$. Hence by replacing $n$ with $nq$, we can deduce that $n(K_X+B)\sim 0$ for some $n$ depending only on $d,\mathfrak{R}$.
\end{proof}

\newpage
Now we move on to prove some results we claimed in low dimension.
\subsection{Proof of Theorem \ref{thm-lc-index-dim-3-4}}
\begin{proof}[Proof of Theorem \ref{thm-lc-index-dim-3-4}]
	Firstly, it suffices to show that Conjecture \ref{conj-bound-dlt-index} hold in dimension 3.  Also by taking dlt modification, we can assume $(X,B)$ is dlt and $B>0$. Therefore there exists $\epsilon>0$ depending only on $\mathfrak{R}$ such that $B\geq \epsilon$. Now consider running an MMP on $K_X+B-\epsilon B$, which will terminate in a Mori Fiber space $X\rightarrow V$ (possibly replacing X). Now if $\dim V$ is zero, then $X$ is fano type and we are done by boundedness of complements. If $\dim V=2$, then by similar arguments as in the above proof, we are done. Hence the only remaining case is when $\dim V=1$. \\\\
	Firstly, by canonical bundle formula, we can write $$q(K_X+B)\sim qf^*(K_V+B_V+M_V)\sim_\Q 0$$ where $qM_V,qB_V$ is Cartier (since V is a smooth curve) and $q$ depend only on $\mathfrak{R}$ . Now if $V$ is rational curve, then we are done again by boundedness of complements as the canonical index of $K_V+B_V+M_V$ is bounded. Hence the remaining case is when $V$ is elliptic curve. In this case, $B_V=0$ and $qM_V$ must be torsion divisor. Although we don't know about the torsion index, we can see that $q(K_X+B)$ is indeed Cartier as it is pull back of Cartier divisor. Now if $(X,B)$ is lc but not klt, then we are done by Theorem \ref{thm-lc-index-bounded}. Hence we can assume $(X,B)$ is klt. Then since Cartier index of $(X,B)$ is bounded, we see that $(X,B)$ is $\delta$-lc for some $\delta$ depending only on $\mathfrak{R}$. Now we are ready to apply Theorem \ref{thm-B-bounded-CY} on log Calabi-Yau fibration, and see that $(X,B)$ is log bounded (since all elliptic curves are bounded). Therefore, we get that the canonical index is indeed bounded.

\end{proof}
\subsection{Proof of Theorem \ref{thm-index-4}}
\begin{proof}[Proof of Theorem \ref{thm-index-4}]
	Since $B\neq 0$, we can run mmp on $K_X$, replacing $X$, we will end with a Mori Fiber space, $f: X\rightarrow V$. where $X$ is fano type over $V$ and $-K_X$ is ample over $V$. If $\dim V$=3, then we are done by canonical bundle formula: there exists $q$ depending only on $\mathfrak{R}$, such that $q(K_X+B)\sim f^* (q(K_V+B_V))$ and  that $q(K_V+B_V)\sim 0$ by theorem \ref{thm-lc-index-dim-3-4}.\\\\ If $\dim V=0$, then $X$ is fano type and the result follows from Theorem \ref{t-bnd-compl}.\\\\ If $\dim V =1$, then either $V$ is rational curve or $V$ is elliptic curve. In both case, we can write $q(K_X+B)\sim qf^*(K_V+B_V+M_V)\sim_\Q 0$, where $qM_V,qB_V$ is Cartier (since V is a smooth curve) and $q$ depend only on $\mathfrak{R}$. If $V$ is rational curve, then again by Theorem \ref{t-bnd-compl}, possibly replacing $q$, we can assume that $q(K_V+B_V+M_V)\sim 0$ and hence we get $q(K_X+B)\sim 0$. If $V$ is elliptic curve, then $V$ is in a bounded family and again we have $q(K_V+B_V+M_V)$ is Cartier, hence $q(K_X+B)$ is Cartier, hence $(X,B)$ is $\epsilon$-lc for $\epsilon$ depending only on $\mathfrak{R}$. Therefore, we can apply \ref{thm-B-bounded-CY}, and conclude that $(X,B)$ belong to a bounded family and hence the result follows. Therefore, the only case left is when $\dim V=2$ and that is the next lemma.
\end{proof}
\begin{prop}
	Let $\mathfrak{R}$ be a finite subset of rational numbers. Then there exists $n$ depending only on $\mathfrak{R}$ satisfying the following: If $(X,B)$ be a dimension 4 klt pair such that $K_X+B\sim_\Q 0$. Assume there is a contraction projective morphism $f:X\rightarrow V$ such that $-K_X$ is ample over $V$ and $V$ is a a surface. Then $n(K_X+B)\sim 0$.
\end{prop}

\begin{proof}
	Firstly let $F$ be the general fiber and let $K_F+B_F := (K_X+B)|_F$. We see that $-K_F$ is ample and $K_F+B_F\sim_\Q 0$. Furthermore, the coefficients of $B_F$ belong to a finite set and $(F,B_F)$ is klt. Hence by BAB, $(F,B_F)$ belong to some bounded family depending only on $\mathfrak{R}$. \\\\
	Now by canonical bundle formula and the fact that $X$ is fano type over $V$, we can find $q$ depending only on $\mathfrak{R}$ such that $$q(K_X+B)\sim qf^*(K_V+B_V+M_V)$$ where $qB_V$ and $qM_V$ are integral divisors and $K_V+B_V+M_V\sim_\Q 0$. We note that since $V$ is a surface, then $M_V$ is $\Q$-Cartier and $(V,B_V)$ is klt. \\\\ 
	Now we break into the following cases, \begin{enumerate}
		\item $M_V\equiv 0$: In this case, since $(F,B_F)$ belong to a log bounded family, by [\cite{floris}, Thm 1.3], there exists $m$ depending only on $\mathfrak{R}$ such that $mM_V\sim 0$, and hence, replacing $q$ by a bounded multiple, we have $qM_V\sim 0$ and $q(K_V+B_V)\sim 0$. Hence we get $q(K_X+B)\sim 0$, as needed.
		
		\item $M_V$ is not numerically trivial, then $K_V$ is not pseudo-effective, (we note that since $V$ is a surface, $M_V$ is $\Q$-Cartier). Hence we can run a MMP on $K_V$, which will terminate with a Mori-Fiber space $g:V\rightarrow W$ with $-(K_V+B_V)$ ample over $W$. Again, we have two cases, if $W$ is a point, then $V$ is Fano type and hence by Theorem \ref{t-bnd-compl}, by replacing $q$ by a bounded multiple, we can assume $q(K_V+B_V+M_V)\sim 0$, hence the result follows. Therefore, we can assume that $W$ is a curve. Now, let $G$ be the general fiber of $g:V\rightarrow W$, we see that $G$ is a rational curve and we can assume that $q(K_G+B_G+M_G)\sim 0$. Hence we have, by generalised canonical bundle formula, $q(K_V+B_V+M_V)\sim qg^*(K_W+B_W+M_V)$, and by replacing $q$, we can assume that $qB_W, qM_W$ is integral hence Cartier since $W$ is a smooth curve. This implies that $q(K_W+B_W+M_W)$ is Cartier, hence $q(K_V+B_V+M_V)$ is Cartier and hence $q(K_X+B)$ is Cartier. Hence there exists $\epsilon$ depending only on $\mathfrak{R}$ such that $(X,B)$ is $\epsilon$-lc. \\Now observe that $X\rightarrow W$ factors as a tower of Mori-Fiber space. And since W is either a rational curve or elliptic curve, it belongs to a bounded family. Hence we are ready to apply [\cite{B-18}, Theorem 1.4] to deduce that $(X,B)$ belong to a bounded family depending only on $\mathfrak{R}$. Hence in particular, we have $n(K_X+B)\sim 0$ for some $n$ depending only on $\mathfrak{R}$ as claimed.
	\end{enumerate}
\end{proof}

Finally we prove Theorem \ref{thm-1.9}.
\begin{proof}[Proof of Theorem \ref{thm-1.9}]
	This Theorem follows from \cite{Xu-2018}, Section 6.
\end{proof}
\newpage
\subsection{Proof of Main Result (Theorem \ref{thm-induction-index})}
\begin{proof}[Proof of Theorem \ref{thm-induction-index}]
	This follows from the inductive statements in Theorem \ref{thm-klt-index-induction}, \ref{thm-lc-index-bounded} and \ref{thm-1.9}.
\end{proof}
This gives immediately Theorem \ref{thm-lc-index-dim-3-4} and Theorem \ref{thm-index-4} as Corollary.
\begin{proof}[Proof of Theorem \ref{thm-lc-index-dim-3-4}]
	This follows directly from Theorem \ref{thm-induction-index} when $d=3$.
	
\end{proof}
\begin{proof}[Proof of Theorem \ref{thm-index-4}]
    This follows from Theorem \ref{thm-lc-index-dim-3-4} and Theorem \ref{thm-klt-index-induction}.
\end{proof}
\newpage

\newpage

\end{document}